\documentclass[12pt]{amsart}

\usepackage{amsmath,amsfonts,amsthm,amsopn,cite,mathrsfs}
\usepackage{amssymb}
\usepackage{epsfig,verbatim}
\newcommand{\up}{uncertainty principle}
\newcommand{\tfa}{time-frequency analysis}

\newcommand{\tf}{time-frequency}

\newcommand{\fif}{if and only if}
\newcommand{\tfs}{time-frequency shift}

\newcommand{\modsp}{modulation space}

\newtheorem{tm}{Theorem}[section]
\newtheorem{lemma}[tm]{Lemma}
\newtheorem{prop}[tm]{Proposition}

\newcommand{\rems}{\noindent\textsl{REMARKS:}}

 \theoremstyle{definition}

\newcommand{\beqa}{\begin{eqnarray*}}
\newcommand{\eeqa}{\end{eqnarray*}}

\newcommand{\field}[1]{\mathbb{#1}}
\newcommand{\bR}{\field{R}}        %  real numbers
\newcommand{\bN}{\field{N}}        %  natural numbers 
\newcommand{\bZ}{\field{Z}}        %  whole numbers 
\newcommand{\bC}{\field{C}}        %  complex numbers
 \def\cB{\mathcal{B}}
 
 \def\cG{\mathcal{G}}

 \def\cC{\mathcal{C}}
 
 \def\cO{\mathcal{O}}

\def\rd{\bR^d}

\def\rdd{{\bR^{2d}}}
\def\zdd{{\bZ^{2d}}}

\def\lrd{L^2(\rd)}

\def\intrdd{\int_{\rdd}}

\def\<{\left<}
\def\>{\right>}

\def\inv{^{-1}}

\def\mv1{M_v^1}

\newcommand{\vs}{\vspace{3 mm}}

\newcommand{\gablam}{\cG (g, \Lambda )}
\newcommand{\gablamn}{\cG (g, \Lambda _n )}
\newcommand{\ggf}{_{g,\Lambda}}
\newcommand{\kor}{\textrm{ker}\, D_{g,\Lambda}}

\begin{document}
\begin{abstract}
We investigate finite sections of Gabor frames and study the
asymptotic behavior of  their lower
Riesz bound. From a numerical point of view, these sets of  \tfs s
are linearly dependent, whereas from a  rigorous analytic point of view, they are conjectured to be
linearly independent. 
\end{abstract}

\newcommand{\cvi}{\mathcal{C}_v ^\infty}
\newcommand{\cv}{\mathcal{C}_v}

\title{Linear Independence of Time-Frequency Shifts?}
\author{Karlheinz Gr\"ochenig}
\address{Faculty of Mathematics \\
University of Vienna \\
Oskar-Morgenstern-Platz 1 \\
A-1090 Vienna, Austria}
\email{karlheinz.groechenig@univie.ac.at}
\date{}
\thanks{K.\ G.\ was
  supported in part  by the  project P26273 - N25  of the
 Austrian Science Foundation (FWF) and 
by the project  ICT 10-066 ``NOWIRE'' of the
Vienna Science and Technology Fund (WWTF)}
% \thanks{K.\ G.\ was
%   supported in part by the  National Research Network S106 SISE of the
% Austrian Science Foundation (FWF)} 
\maketitle

\section{Introduction}

A famous  conjecture of Heil, Ramanathan, and Topiwala~\cite{HRT96},
often called the HRT-conjecture,  states that finitely
many \tfs s of a non-zero  $L^2$-function are linearly independent. 
Denoting a time-frequency shift of $g\in \lrd $  along  $z=(x,\xi ) \in \rdd $ by 
$$
\pi (z) g(t) = M_\xi T_x g(t) = e^{2\pi i \xi\cdot t} g(t-x), \qquad
\qquad t\in \rd\, ,
$$
the question is whether 
$$
\sum _{j=1}^n c_j \pi (z_j) g = 0    \quad \Longrightarrow c_j = 0
\qquad \forall j \, , 
$$
for arbitrary points  $z_1, \dots , z_n\in  \rdd $. 
    
To this day this   conjecture is open, it is known to be true only
under restrictive conditions on either $g$ or  the set $\{z_j \}$.

(a) Linnell's Theorem~\cite{Lin99}: Let $\Lambda \subseteq \rdd $ be a lattice and
$g\in \lrd $ arbitrary, then for every finite subset $F\subseteq \Lambda $ the
set $\{\pi (\lambda ) g: \lambda \in F\}$ is linearly
independent. % This means that for every finite subset $F\subseteq
% \Lambda $ there exist constants $A_F, B_F>0$ (depending on $F$) such
% that 
% \begin{equation}
%   \label{eq:1}
%   A_F \|c\|_2^2 \leq \| \sum _{\lambda \in F} c_\lambda \pi (\lambda
%   )g\|_2^2 \leq B_F \|c\|_2^2 , \quad \text{ for all } c\in \ell
%   ^2(F ) \, .
% \end{equation}
This is a deep result obtained with  von Neumann algebra techniques;  special cases have been
 reproved with  more analytic arguments in~\cite{BS10,DG13a}. 

(b) Bownik and Speegle ~\cite{BS13} proved the HRT-conjecture for
$g$ with one-sided  super-exponential decay. This result contains the
early results of~\cite{HRT96}.    % or if the ambiguity
% function of $g$ decay exponentially ($g$ is in some Gelfand-Shilov
% space), the  conjecture holds for arbitrary points $\{z_1, \dots
% z_n\}$ by a recent result of Bownik and Speegle;

In view of these general results, it is rather surprising that  it is
not  known whether  four  arbitrary \tfs s of $g\in \lrd
$  are linearly independent.  Even for rather  special
constellations the linear independence of four \tfs s is highly
non-trivial~\cite{DZ12}. 

Further contributions to the HRT-conjecture investigate the kernel of a linear combination of \tfs\
operators~\cite{balan08} and estimates of the  frame  bounds of 
finite sets of \tfs s~\cite{CL01}. 

For a detailed survey of the linear independence conjecture we refer
to Heil's article~\cite{heil06}.

In this note we adopt  a different point of view and investigate  the
numerical linear independence of \tfs s. In other words, can we
determine numerically whether a given finite set of \tfs s is linearly
independent?  We will argue that the  answer is negative.  To formulate
a precise result, we will study  the lower Riesz bound
of finite sections of a Gabor frame  and estimate its asymptotics. 
By taking larger and larger finite sections, the lower Riesz converges
to zero, and in many cases this convergence is super-fast. Thus from a
numerical point of view  even
 small sets of \tfs s may look linearly dependent. 
 The main
result will illustrate the spectacular difference between a conjectured
mathematical truth and a computationally observable  truth. 

Let us explain the problem in detail. Let $\lambda = (\lambda _1,
\lambda _2) \in \rd \times \rd \simeq \rdd $ be a point in the \tf\
plane (or phase space in the terminology of physics). The \tfs\ $\pi
(\lambda )$ acts on a function $g\in \lrd $ by 
$$
\pi (\lambda ) f(t) = e^{2\pi i \lambda _2\cdot t} g(t-\lambda _1) \,
.
$$
For fixed $g \in \lrd $ and  a countable subset $\Lambda \subseteq
\rdd $, the set $\gablam = \{ \pi  (\lambda ) g:
\lambda \in \Lambda \}  $ is called a Gabor system, and for $n>0$ the
set 
$$
\gablamn = \cG (g, \Lambda \cap B_n(0)) = \{ \pi (\lambda ) g:
\lambda \in \Lambda , |\lambda | \leq n\}
$$
is a \emph{finite section } of $\gablam $. 
We are interested in the quantity
\begin{equation}
  \label{eq:10}
  A_n = A(g, \Lambda _n) = \min _{c\neq 0} \frac{\|\sum _{|\lambda | \leq
    n} c_\lambda \pi (\lambda ) g\|_2^2}{\sum _{|\lambda |\leq n}
  |c_\lambda |^2} \, .
\end{equation}
Since $\gablamn$ spans a finite-dimensional subspace of $\lrd $, the
minimum exists. Moreover, $A_n=
0$, \fif\  $\gablamn $ is linearly dependent. Thus we may take $A_n$ as a quantitative measure for the numerical
linear dependence  of $\gablamn $.

Our main result is an asymptotic estimate for $A_n$ as $n\to \infty
$.  Before formulating this estimate, we need to explain some of the
basic concepts of Gabor analysis and \tfa . 
We refer to the textbooks~\cite{chr03,book,heil11} for detailed
expositions of \tfa\ and frame theory.

A Gabor system $\gablam $ is a frame, a so-called Gabor frame, if
there exist frame bounds $A,B>0$, such that 
$$
A \|f\|_2^2 \leq  \sum _{\lambda \in \Lambda } |\langle f, \pi
(\lambda ) g\rangle |^2  \leq B \|f\|_2^2  \quad \forall f\in \lrd \,
.
$$
For an equivalent and more suitable condition we define the synthesis
operator $D_{\ggf }  $ 
$$
D_{\ggf } c = \sum _{\lambda \in \Lambda } c_\lambda \pi (\lambda )g \, ,
$$
which is well-defined  on  finite sequences $c$. 
Then $\gablam $ is a frame, \fif\ $D_{\ggf } : \ell ^2(\Lambda ) \to \lrd
$ is bounded and onto $\lrd $. 

 If, in addition to the frame property,   $\kor = \{0\}$, then $\gablam $
is a Riesz basis for $\lrd $.  In this case there exist $A',B' >0$,
such that 
$$
A' \|c\|_2^2 \leq \|\sum _{\lambda \in \Lambda } c_\lambda \pi
(\lambda )g\|_2^2 \leq B' \|c\|_2^2 \quad \forall c\in \ell ^2(\Lambda
) \, .
$$
In other words, a Riesz sequence  is $\ell ^2$-linearly
independent. In particular, every finite subset of a Riesz sequence
$\gablam $ is linearly independent. 

If $\gablam $ is a frame, but not a Riesz basis,  then by definition
$\kor \neq \{0\}$. However,  if the linear
independence conjecture  is true, then certainly $\mathrm{ker}\, D_{g,
  \Lambda _n} = \{ 0 \}$ for all $n\in \bN $. This means that for
$n\to \infty $,  the finite sets $\gablamn $ must
 get ``more and more linearly dependent''. Quantitatively,  this means
 that the lower Riesz bound  $A_n $ must 
 tend to $0$. 

Our main theorem  shows that this transition to linear dependence  may happen
 very fast. 

\begin{tm} \label{tm-main}
Let $v: \rdd \to \bR ^+$ be a submultiplicative weight function such
that  $\lim _{n\to \infty } v(nz)^{1/n} =1 $ for all
$z\in \rdd $ ($v$ satisfies the Gelfand-Raikov-Shilov condition).

  Assume that
  \begin{equation}
    \label{eq:11}
    \intrdd |\langle g, \pi (z) g\rangle | v(z) \, dz < \infty \, .
  \end{equation}
 % $g\in M^1_v(\rd )$ for some submultiplicative weight $v$
 %  satisfying the GRS-condition.
If  $\gablam $ is a frame for $\lrd $,  but not   a Riesz basis, then
the lower Riesz bound $A_n$ of $\gablamn$ decays like 
\begin{equation}
  \label{eq:3}
  A_n \leq C  \sup _{|\lambda | >n } v(\lambda )^{-2} \, .
\end{equation}
\end{tm}

For the polynomial weight $v(z) = (1+|z|)^s$,  the lower bound decays 
like $A_n  = \cO ( n^{-2s} ) $, and for the sub-exponential weight 
$v(z) = e^{a|z|^b}$ for $a>0$ and $0<b<1$ we have $A_n = \cO (
e^{-an^b}) $. This means that the lower bound $A_n $ tends to zero
almost exponentially. The finite Gabor system $\gablamn$ is extremely
badly conditioned,  and numerically $\gablamn $ behaves like a linearly
dependent set. On the other hand, 
if $\Lambda $ is a lattice, then by  Linnell's theorem  $\gablamn$ is
always linearly independent. Theorem~\ref{tm-main} states a striking
contrast between  the
numerical linear dependence of finite sets of \tfs s and their conjectured abstract  linear independence. 

\vs

In the remainder of this note we prepare the necessary background on
\tfa\ and spectral invariance of matrix algebras and then prove
Theorem~\ref{tm-main} and a  variation. The proof will be
relatively short, but it combines several non-trivial 
statements from harmonic analysis. In a sense, we extend the
quantitative analysis of the finite section method in~\cite{GRS10} to
elements in the kernel of a matrix. 

\vs

\textbf{Operators related to Gabor systems.} If $D_{\ggf } $ is bounded from $\ell ^2(\Lambda ) $ to $\lrd $, then
$\gablam $ is  called a Bessel sequence. Its adjoint operator
is the analysis operator $D^*_{\ggf } f = \big( \langle f, \pi
(\lambda ) g\rangle : \lambda \in \Lambda \big)\in \ell ^2(\Lambda ) $
for $f\in \lrd $. 

We also consider the frame
operator of $\gablam $ defined to be
\begin{equation}
  \label{eq:12}
  S_{\ggf } f = D_{\ggf }  D_{\ggf } ^* f = \sum _{\lambda \in \Lambda } \langle
  f, \pi (\lambda ) g\rangle \pi (\lambda ) g \,% \qquad \text{ for } f\in \lrd \, ,
\end{equation}
for $f$ in  a suitable space of test functions. The  Gram matrix is 
the matrix   $ G_{\ggf }  =  D_{\ggf } ^*  D_{\ggf }$ acting on $\ell
^2(\Lambda )$  with entries 
$$
(G_{\ggf } )_{\lambda , \mu } = \langle
\pi (\mu )g, \pi (\lambda )g\rangle  \, \qquad \lambda , \mu \in
\Lambda \, .
$$

% $$
% G_{\ggf } c =  D_{\ggf } ^*  D_{\ggf }  c = \sum _{\mu \in \Lambda } \langle
% \pi (\mu )g, \pi (\lambda )g\rangle c_\mu 
% $$
% defined on finitely supported sequences. 

The algebraic identity 
\begin{align*}
  \| \sum _{|\lambda | \leq n } c_\lambda \pi (\lambda
  )g\|_2^2 = \sum _{|\lambda | , |\mu | \leq n} \langle \pi (\lambda
  ) g, \pi (\mu )g\rangle c_\lambda \overline{c_\mu } 
\end{align*}
shows that the Riesz bounds of $\gablamn$ are just the extremal
eigenvalues of the finite sections of the  Gramian matrix of $\gablam
$. 

\vs

\textbf{Weights and \modsp s.} To measure the \tf\ concentration of a
function, we use weighted \modsp s. In \tfa\ 
one uses  the several   conditions for weight  functions~\cite{gro07c}: \\ 
(i) a weight  $v: \rdd \to \bR ^+$ is \emph{submultiplicative}, if  $v(z_1+z_2) \leq 
v(z_1) v(z_2) $ for all $z_1,z_2\in \rdd $, and \\
(ii) $v$ is  \emph{subconvolutive}, if $(v\inv \ast v\inv )(z) \leq C
v(z)\inv $ for all $z\in \rdd $. \\
(iii) A weight $v$ satisfies the \emph{Gelfand-Raikov-Shilov (GRS) condition}
$$\lim _{n\to \infty } v(nz)^{1/n} =1  \quad \text{ for all } z\in
\rdd  \, .
$$ 
The main examples for weights  are the polynomial weights $z\mapsto 
(1+|z|)^s$ for $s\geq 0$  and the sub-exponential weights $z\mapsto 
e^{a|z|^b}$ for $a>0$ and $0<b<1$. The exponential weight $z\mapsto 
e^{a|z|}$ for $a>0$ does not satisfy the GRS-condition.

Let $\phi (t) = e^{-\pi t^2}$ be the Gaussian and $v$  a weight
function  on $\rdd $.  A function $g$ belongs to the  \modsp\
$M^1_v(\rd )$, if
$$
\|g\|_{M^1_v} := \intrdd |\langle g, \pi (z) \phi \rangle | \, v(z) \, dz < \infty \,
.
$$
 Likewise
$g\in M^\infty _v(\rd )$, if   
$$ \|g\|_{M^\infty _v} := \sup _{z\in \rdd } |\langle g, \pi (z) \phi \rangle | \, v(z) <
\infty \, .
$$
From  the  theory of \modsp s we need the following facts
about the \modsp s $M^1_v$ and $M^\infty _v$. See  \cite{book} and
\cite{feiSTSIP} for a historical survey about \modsp s.

\begin{lemma} \label{amalg}
   (A)  Assume that $v$ is a submultiplicative weight on $\rdd $.  Then
the following conditions are equivalent:

(i)  $g\in M^1_v (\rd )$ 

(ii)   $\intrdd
|\langle g, \pi (z) g \rangle | \, v(z) \, dz < \infty $. 

(iii)  The function $z\mapsto  \langle g, \pi (z) g\rangle $
belongs to the amalgam space $W(C, \ell ^1_v)$, i.e., it is continuous and  
\begin{equation}
  \label{eq:c7}
\sum _{k\in \zdd } \sup _{z\in [0,1]^{2d}} |\langle g, \pi
(k+z)g\rangle |  v(k) < \infty \, .  
\end{equation}

(B)  Assume that  $v$ is submultiplicative and subconvolutive.  Then $g\in M^\infty _v(\rd )$ \fif\
$ \sup _{z\in \rdd }|\langle g, \pi (z) g \rangle | \, v(z) <\infty $. 
 \end{lemma}

For a proof see~\cite{book}, Propositions~12.1.2, 12.1.11 and
Theorem~13.5.3.  %, especially ***

Note that condition~\eqref{eq:11} in 
Theorem~\ref{tm-main} amounts to saying that $g\in M^1_v(\rd )$. 

\vs

\textbf{Spectral invariance of matrices with off-diagonal decay.} 
Let $\Lambda $ be a countable set in $\rdd $ satisfying the condition
$$
\max _{z\in \rdd }  \# \{ \lambda \in \Lambda : |\lambda - z | \leq
1\} < \infty \, .
$$
$\Lambda $ is said to be relatively separated. Let $v$ be a
submultiplicative weight on $\rdd $. 

We will use the following classes of infinite matrices over the index set  $\Lambda $. 

(i) The class $\cC _v^\infty (\Lambda )$ consists of  matrices $A =
(a_{\lambda \mu })_{\lambda ,\mu \in \Lambda }$  with
\emph{off-diagonal decay} $v\inv $ and is equipped with  the norm
\begin{equation}
  \label{eq:5}
  \|A\|_{\cvi } =  \sup _{\lambda ,\mu \in \Lambda } |a_{\lambda \mu }|
  v(\lambda -\mu ) \, .
\end{equation}
For polynomials weights $v(z) = (1+|z|)^s$, $\cvi $  is often called the
Jaffard class.

(ii)  A matrix $A $ belongs to 
the class $\cv = \cv (\Lambda )$ of \emph{convolution-dominated matrices}, if there
exists an envelope function $\Theta \in W(C,\ell ^1_v)$, such that
$$
|a_{\lambda \mu } | \leq \Theta (\lambda -\mu ) \qquad \forall \lambda
, \mu \in \Lambda \, .
$$
The norm on $\cv $ is $\|A\|_{\cv } = \inf \{ \|\Theta \|_{W(C,\ell
  ^1_v)} : \Theta \,\, \text{ is an envelope }\}$.  
% \begin{equation}\label{eq:7}
%   \|A\|_{\cv } = \sum _{\mu \in \Lambda} \sup _{\lambda  \in \Lambda}
%   |a_{\lambda , \lambda  - \mu } |  \,  v(\mu ) \, .
% \end{equation}

If $v$ is submultiplicative, then $\cv $ is a Banach algebra. If $v\inv
\in \ell ^1(\Lambda ) $ and $v $ is subconvolutive, then $\cvi $
is a Banach algebra. Both algebras can be embedded into the
$C^*$-algebra of bounded operators  $\cB( \ell ^2(\Lambda ))$. 

The most important result about these matrix algebras is their
spectral invariance asserting that  the off-diagonal decay is preserved under
inversion. 

\begin{tm}\label{spec}
  Assume that $\Lambda $ is relatively separated and the $v$ is a
  submultiplicative weight satisfying the GRS-condition. 

(i) If $A\in \cv $ and $A $ is invertible on $\ell ^2(\Lambda )$, then
$A\inv \in \cv $.

(ii) Assume in addition that $v$ is subconvolutive. If $A\in \cvi $
and $A $ is invertible on $\ell ^2(\Lambda )$, then 
$A\inv \in \cvi $. 
\end{tm}

We say that both $\cv $ and $\cvi $  are inverse-closed in $\cB (\ell
^2(\Lambda ))$. Theorem~\ref{spec} has been proved several times and
on several levels of generality. We refer to the original work  of
Baskakov~\cite{Bas90}, Kurbatov~\cite{Kur90},
Gohberg-Kaeshoek-Woerdemann~\cite{GKW89}, and Sj\"ostrand~\cite{Sjo95} for (i),
and to Baskakov~\cite{Bas90}, Jaffard~\cite{jaffard90}, and
~\cite{GL04a} for (ii). 
 The attributions for the algebra $\cv$
are a  bit subtle, because the cited references deal only with the case
when $\Lambda $ is a lattice. The case of a  relatively separated
index set $\Lambda $ follows by a simple reduction described
in~\cite{BCHL06a}: Since $\max \# \big(\Lambda \cap (k+[0,1]^{2d})\big) = N
<\infty$, one  can  define an explicit  map $a: \Lambda \mapsto  \zdd
$  that preserves the off-diagonal decay properties after re-indexing
a given matrix $A$. For the spectral invariance one may  assume
therefore without loss of generality that $\Lambda $
is a lattice. Also, Sj\"ostrand's argument~\cite{Sjo95} works 
for   relatively separated index sets and weights without any
change of the proof.    
An extended survey about spectral invariance  including matrix algebras
can be found in ~\cite{Gr10}.

These matrix classes arise naturally in the analysis of Gabor frames,
as is shown by the following lemma. 
\begin{lemma} \label{mat}
Assume that $\Lambda \subseteq \rdd $ is relatively separated and that
$v$ is a submultiplicative weight on $\rdd $. 

(i) If $g\in M^1_v(\rd )$, then the Gramian $G_{\ggf } $ of $\gablam $
is in $\cv (\Lambda )$. 

(ii) If, in addition, $v$ is subconvolutive and if  $g\in M^\infty _v(\rd
)$, then $G_{\ggf } \in \cvi (\Lambda ) $. 
\end{lemma}

\begin{proof}
Since  
$$
|(G_{g,\Lambda } )_{\lambda , \mu }| = | \langle \pi (\mu )g, \pi
(\lambda )g\rangle | = |\langle g, \pi (\lambda -\mu )g\rangle | \, ,
$$
we may take $\Theta (z)  = |\langle g, \pi (z) g\rangle | $  as an envelope
function. If  $g\in M^1_v(\rd )$, then $\Theta \in  W(C,\ell ^1_v) $
by Lemma~\ref{amalg}.  (ii) is clear from the definitions.  
\end{proof}

\vs

\textbf{Proof of Theorem~\ref{tm-main}.}
Theorem~\ref{tm-main} follows from the combination of several
observations. First an easy lemma.

\begin{lemma}\label{easy}
  Assume that $\gablam $ is a Bessel sequence with bound $B$ and that
  $\kor \neq \{0\}$. If $c\in \kor , \|c\|_2=1$, then for sufficiently
  large $n$ we have
  \begin{equation}
    \label{eq:c8}
    A_n \leq 2B \sum _{\lambda \in \Lambda : |\lambda |>n} |c_\lambda |^2 \, .
  \end{equation}
\end{lemma}

\begin{proof}
We split the sum  $\sum
_{\lambda \in \Lambda } c_\lambda \pi (\lambda )g = 0$ into  two parts
and then take norms. We obtain 
\begin{align*}
  \|\sum _{|\lambda | \leq n} c_\lambda \pi (\lambda )g\|_2^2&=
  \|\sum _{|\lambda | > n} c_\lambda \pi (\lambda )g\|_2^2  
\leq B \sum _{|\lambda | >n} |c_\lambda |^2  \, . % =  \notag \\
% &= B \frac{\sum _{|\lambda | > n} |c_\lambda |^2}{\sum _{|\lambda | \leq
%   n} |c_\lambda |^2}  \sum _{|\lambda | \leq n} |c_\lambda |^2  \notag \\
% &\leq 2 B_g  \sum _{|\lambda | > n} |c_\lambda |^2  \, . \label{decay}
\end{align*}
For $n$ large enough we have $\sum _{|\lambda |\leq n} |c_\lambda |^2 \geq
\tfrac{1}{2}  $, whence the lower Riesz bound $A_n$ of
$\gablamn$ obeys the  following
estimate:
\begin{equation}
  \label{eq:13}
A_n = \inf _{c\neq 0} \frac{\| \sum _{|\lambda |\leq  n} c_\lambda \pi
  (\lambda )g \|_2^2}{\sum _{|\lambda |\leq n} |c_\lambda | ^2} \leq 
2B   \sum _{|\lambda | > n} |c_\lambda |^2\, .
\end{equation}
This estimate holds for every normalized $c\in  \mathrm{ker}\, D_{g,\Lambda } $. 
  \end{proof}

Lemma~\ref{easy} states the obvious fact that the finite sets
$\gablamn$ become ``more and more linearly dependent'' in the sense
that $A_n \mapsto  0$. 
To estimate the asymptotic behavior of $A_n$ more precisely, we need to construct a  ``bad''
sequence $c$ with fast decay in  $\mathrm{ker}\, D_{g,\Lambda }$. The
possible decay depends on the \tf\ concentration of the window $g$, as
we will prove now. 

\begin{prop} \label{ker}
  If $g\in M^1_v(\rd )$ and $\gablam $ is a frame, but not a Riesz
  basis for $\lrd $,  then
  $\mathrm{ker}\, D_{g,\Lambda } \cap \ell ^1_v(\Lambda ) \neq
  \{0\}$. 
\end{prop}

\begin{proof}
  1. Recall that $G_{g,\Lambda } = D_{g,\Lambda } ^* D_{g,\Lambda }$
  is the Gramian operator associated to $\gablam $. 
Consequently, $c\in \mathrm{ker} \, D_{g,\Lambda } $ \fif\
$\|D_{g,\Lambda } c\|_2^2 = \langle G_{g,\Lambda } c,c \rangle = 0$
\fif\ $c\in \mathrm{ker} \, G_{g,\Lambda } $. \\

2. To relate the spectrum of the frame operator $S_{\ggf } $ on $\lrd
$ and of $G_{\ggf } $ on $\ell ^2(\Lambda )$, we use the identity 
$$
\sigma (S_{g,\Lambda }) \cup \{0\} = \sigma (D_{g,\Lambda }
D^*_{g,\Lambda }) \cup \{0\} =   \sigma (D^*_{g,\Lambda }
D_{g,\Lambda }) \cup \{0\} = \sigma (G_{g,\Lambda }) \cup \{0\} \, ,
$$
which follows from a purely  algebraic manipulation~\cite[p.\
199]{conway90}. 

From this identity  we draw the following conclusions: Since $\gablam
$ is a frame, we have   $\sigma
(S_{g,\Lambda }) \subseteq [A,B]$ for $A,B>0$. Since $\gablam $ is not
a Riesz basis, $\mathrm{ker}\, G_{\ggf } \neq \{0\}$ and thus $0\in
\sigma (G_{\ggf} )$. Consequently, 
\begin{equation}
  \label{eq:c9}
  \sigma (G_{\ggf} ) \subseteq \{0\}  \cup [A,B] \, .
\end{equation}
 The main point is the spectral gap between $0$ and $A$. 

3. We now apply an argument developed by Baskakov~\cite{Bas97a}
 to show that the
orthogonal projection onto the kernel of $G_{\ggf } $ is  a matrix
with off-diagonal decay. Let $P$ be the orthogonal projection from
$\ell ^2(\Lambda )$ onto $\mathrm{ker}\, G_{\ggf }$.  With the Riesz
functional calculus~\cite{conway90},  this  projection
 can be written as 
\begin{equation}
  \label{eq:6}
  P  = \frac{1}{2\pi i} \int _\gamma (z I - G_{g,\Lambda
  })\inv \, dz \, ,
\end{equation}
where $\gamma $ is a closed curve in $\bC $ around $0$ disjoint from
the interval $[A,B]$, for instance $\gamma (t) = \frac{A}{2} e^{2\pi i
  t},  t\in [0,1]$. 

4. Spectral invariance: By Lemma~\ref{mat} $G_{\ggf }$ and
$z\mathrm{I} - G_{\ggf }$ are matrices in $ \cC _v $. Since  $zI-
G_{g,\Lambda }$  is  invertible for $z\in \gamma $, Theorem~\ref{spec}
implies that  $(zI - G_{g,\Lambda })\inv  $ is also in $\cC _v$. From
the 
continuity of the resolvent function  $z\mapsto  (zI - G_{g,\Lambda }
)\inv $  we conclude that $\sup _{z\in \gamma } \|(zI-G_{g,\Lambda
})\inv \|_{\cC _v} <\infty $. Consequently,  the integral defining the
orthogonal projection onto the kernel of $G_{g,\Lambda } $  is  
in the algebra of convolution-dominated matrices $\cC _v$: 
$$
P  \in \cC _v \, .
$$
This means that there exists an envelope $\Theta \in W(C,\ell ^1_v)$,
such that $|P_{\lambda \mu } | \leq \Theta (\lambda -\mu )$.  If  $\{ e_\lambda : \lambda \in \Lambda \}$ with  $e_\lambda (\mu ) =
\delta _{\lambda , \mu }$ denotes  the standard orthonormal basis of $\ell ^2(\Lambda
)$, then 
$$
|\langle e_\lambda , Pe_\mu \rangle | = |P_{\lambda , \mu } | \leq \Theta (\lambda -\mu ) \, , 
$$
or, equivalently, $Pe_\mu \in \ell ^1_v (\Lambda )$ for all $\mu \in
\Lambda $. As the projection $P$ is non-zero by assumption, $Pe_\mu
\neq 0$ for some $\mu$, and thus we have 
found a non-trivial vector in $\mathrm{ker}\, G_{\ggf } \cap \ell ^1_v
= \kor \cap \ell ^1_v $, and
we are done.  
\end{proof}

Combining Lemma~\ref{easy} and Proposition~\ref{ker}, 
we now can conclude the proof of Theorem~\ref{tm-main}. 
Choose an $\ell ^2$-normalized  $c\in \mathrm{ker}\, D\ggf \cap \ell ^1_v(\Lambda
)$. 
Then by \eqref{eq:13} we obtain that 
\begin{align}
A_n  &\leq 2  B \sum _{|\lambda | >n} |c_\lambda |^2 \notag \\
&\leq 2 B \sup _{|\lambda | >n} v(\lambda  )^{-2} \,  \sum _{|\lambda
  | >n} |c_\lambda |^2 v(\lambda )^2  \notag \\
& \leq 2 B \sup _{|\lambda | >n} v(\lambda  )^{-2} \,  \sum _{|\lambda
  | >n} |c_\lambda | v(\lambda )  = C \sup _{|\lambda | >n} v(\lambda
)^{-2}  \, . \label{final}
\end{align}
Theorem~\ref{tm-main} is proved completely. \hfill  $\Box$

The same proof yields the following variation of
Theorem~\ref{tm-main}. 

\begin{tm} \label{tm-mainb}
Let $v$ be a submultiplicative and subconvolutive  weight function
satisfying  the Gelfand-Raikov-Shilov condition.

Assume that $g\in M^\infty _v(\rd )$ and that   $\gablam $ is a frame
for $\lrd $,  but not   a Riesz basis.  Then
the lower Riesz bound $A_n$ of $\gablamn$ decays like 
\begin{equation}
  \label{eq:3a}
  A_n \leq C \sum _{|\lambda | >n}  v(\lambda )^{-2} \, .
\end{equation}
\end{tm}

\begin{proof}
  The proof is similar,  we just use the versions of Lemma~\ref{mat}  and
  Theorem~\ref{spec} that are valid for $M^\infty _v(\rd
  )$. Instead of Proposition~\ref{ker} we use the following statement: 
If  $g\in M^\infty _v(\rd )$ and $\gablam $ is a frame, but not a
Riesz basis for $\lrd $, then $\kor \cap \ell ^\infty _v (\Lambda )
\neq \{0\}$.   
Equation~\eqref{final} is replaced by 
\begin{align*}
A_n &\leq 2B \sum _{|\lambda | >n} |c_\lambda |^2 \notag \\
&\leq 2 B \sup _{|\lambda | >n} |c_\lambda |^2  v(\lambda  )^{2} \,  \sum _{|\lambda
  | >n}  v(\lambda )^{-2} \, .
\end{align*}
\end{proof}

\rems\ 1.  Note the  importance of assumptions: $\gablam $ must be a frame so that
there exists a spectral gap for the Gramian. Theorem ~\ref{tm-main} fails,
when $\gablam $ is not a frame and the spectral gap is missing. 
This may be  the case for Gabor systems  at the critical density, for
instance, with $\phi (t) = e^{-\pi t^2}$ the Gabor system $\cG (\phi ,
\bZ ^2)$ is neither a frame nor a Riesz basis (but still complete in
$L^2(\bR )$). In this case, the asymptotic decay of the lower Riesz
bound $A_n$ can  be investigated with different methods,
see~\cite{Ban14}. 

2. Theorem~\ref{tm-main} quantifies the degree of linear dependence of
the finite sets $\gablamn $. Note that good \tf\ localization of $g$ (corresponding to fast
growth of $v$) yields a faster decay of the constants $A_n$. This is
somewhat counter-intuitive, because the fast decay of $z\mapsto   \langle g
, \pi (z) \phi \rangle  $  implies  that  the function $z \mapsto  |\langle
g , \pi (z) \phi \rangle |^2$  is sharply peaked in 
$\rdd $, and shifts of sharply peaked bumps (corresponding to the
\tfs s  of $\pi (\lambda )g$) tend to be linearly independent with
good constants. % Intuitively the quality of
% linear dependence  should
% be higher for sharply peaked functions.
According to Theorem~\ref{tm-main} this is not the case
here. This  phenomenon indicates the existence of
subtle  cancellations in linear combinations of \tfs s  and  seems to
be yet another  manifestation of the \up . 

3. To obtain an upper estimate for $A_n$, we needed to find only a \emph{single}
sequence $c \in \ell ^2(\Lambda )$ such that 
$ \|\sum _{|\lambda |\leq n} c_\lambda \pi (\lambda )g\|_2^2  \approx
A_n \|c\|_2^2$.  In the course of the proof we have
constructed such a sequence by using the spectral invariance and the
properties of the basis function $g$. 

It is natural to ask whether the decay rate of $A_n$ in Theorem~\ref{tm-main}
is best possible.   This question, however, is much more
difficult, because it amounts  to showing that 
$ \|\sum _{|\lambda |\leq n} c_\lambda \pi (\lambda )g\|_2^2 \geq
\mathrm{const} \, A_n
\|c\|_2^2$ for \emph{all} $c$. Since every finite set of \tfs s can be
extended to a Gabor frame,  this statement seems 
equivalent to the original linear independence conjecture.

4. If $v$ is an exponential weight, $v(z) = e^{a|z|}$ for some $a>0$,
then  the matrix algebras $\cv $ and $\cvi $ are no longer
inverse-closed in $\cB (\ell ^2(\Lambda ))$. The statement of Theorem~\ref{spec} is
false and has to be replaced by a weaker version. Nevertheless one can
show~\cite{Ban14} that for $g\in M^1_v$ with exponential weight $v(z)
= e^{a|z|}$  the lower Riesz
bound decays exponentially $A_n \lesssim e^{-\epsilon n}$ for some
$\epsilon >0$. 

5. In our analysis we have only used that $\gablam $ is a frame with
$\kor \neq \{0\}$ and the decay properties of the Gramian $G_{\ggf
}$. The statement about the asymptotic behavior of the lower Riesz
bound $A_n$ carries over without change to general localized
frames~\cite{FoG05} 
indexed by a relatively separated subset of $\rdd $.

\def\cprime{$'$} \def\cprime{$'$} \def\cprime{$'$} \def\cprime{$'$}
  \def\cprime{$'$} \def\cprime{$'$}

 % \bibliographystyle{abbrv}
 % \bibliography{general,new}

\begin{thebibliography}{10}

\bibitem{balan08}
R.~Balan.
\newblock The noncommutative {W}iener lemma, linear independence, and spectral
  properties of the algebra of time-frequency shift operators.
\newblock {\em Trans. Amer. Math. Soc.}, 360(7):3921--3941, 2008.




\bibitem{heil11}
C.~Heil.
\newblock {\em A basis theory primer}.
\newblock Applied and Numerical Harmonic Analysis. Birkh\"auser/Springer, New
  York, expanded edition, 2011.


\bibitem{BCHL06a}
R.~Balan, P.~G. Casazza, C.~Heil, and Z.~Landau.
\newblock Density, overcompleteness, and localization of frames. {I}. {T}heory.
\newblock {\em J. Fourier Anal. Appl.}, 12(2):105--143, 2006.


\bibitem{Ban14}
S.~Bannert
\newblock Aspects of Gabor Analysis. 
\newblock Ph.~D.~Thesis. Univ.\ of Vienna, 2014.

\bibitem{Bas90}
A.~G. Baskakov.
\newblock Wiener's theorem and asymptotic estimates for elements of inverse
  matrices.
\newblock {\em Funktsional. Anal. i Prilozhen.}, 24(3):64--65, 1990.

\bibitem{Bas97a}
A.~G. Baskakov.
\newblock Asymptotic estimates for elements of matrices of inverse operators,
  and harmonic analysis.
\newblock {\em Sibirsk. Mat. Zh.}, 38(1):14--28, i, 1997.



\bibitem{BS10}
M.~Bownik and D.~Speegle.
\newblock Linear independence of {P}arseval wavelets.
\newblock {\em Illinois J. Math.}, 54(2):771--785, 2010.


\bibitem{BS13}
M.~Bownik and D.~Speegle.
\newblock Linear independence of time-frequency translates of functions with
  faster than exponential decay.
\newblock {\em Bull. Lond. Math. Soc.}, 45(3):554--566, 2013.

\bibitem{chr03}
O.~Christensen.
\newblock {\em An introduction to frames and {R}iesz bases}.
\newblock Applied and Numerical Harmonic Analysis. Birkh\"auser Boston Inc.,
  Boston, MA, 2003.

\bibitem{CL01}
O.~Christensen and A.~M. Lindner.
\newblock Lower bounds for finite wavelet and {G}abor systems.
\newblock {\em Approx. Theory Appl. (N.S.)}, 17(1):18--29, 2001.

\bibitem{conway90}
J.~B. Conway.
\newblock {\em A course in functional analysis}.
\newblock Springer-Verlag, New York, second edition, 1990.


\bibitem{DZ12}
C.~Demeter and A.~Zaharescu.
\newblock Proof of the {HRT} conjecture for {$(2,2)$} configurations.
\newblock {\em J. Math. Anal. Appl.}, 388(1):151--159, 2012.


\bibitem{DG13a}
C.~Demeter and S.~Z. Gautam.
\newblock On the finite linear independence of lattice {G}abor systems.
\newblock {\em Proc. Amer. Math. Soc.}, 141(5):1735--1747, 2013.

\bibitem{feiSTSIP}
H.~G. Feichtinger.
\newblock Modulation spaces: looking back and ahead.
\newblock {\em Sampl. Theory Signal Image Process.}, 5(2):109--140, 2006.

\bibitem{FoG05}
M.~Fornasier and K.~Gr{\"o}chenig.
\newblock Intrinsic localization of frames.
\newblock {\em Constr. Approx.}, 22(3):395--415, 2005.

\bibitem{GKW89}
I.~Gohberg, M.~A. Kaashoek, and H.~J. Woerdeman.
\newblock The band method for positive and strictly contractive extension
  problems: an alternative version and new applications.
\newblock {\em Integral Equations Operator Theory}, 12(3):343--382, 1989.

\bibitem{book}
K.~Gr{\"o}chenig.
\newblock {\em Foundations of time-frequency analysis}.
\newblock Birkh\"auser Boston Inc., Boston, MA, 2001.


\bibitem{gro07c}
K.~Gr\"ochenig.
\newblock Weight functions in time-frequency analysis.
\newblock In e.~a. L.~Rodino, M.-W.~Wong, editor, {\em Pseudodifferential
  Operators: Partial Differential Equations and Time-Frequency Analysis},
  volume~52, pages 343 -- 366. Fields Institute Comm., 2007.

\bibitem{Gr10}
K.~Gr\"ochenig.
\newblock Wiener's lemma: Theme and variations. an introduction to spectral
  invariance.
\newblock In B.~Forster and P.~Massopust, editors, {\em Four Short Courses on
  Harmonic Analysis}, Appl. Num. Harm. Anal. Birkh\"auser, Boston, 2010.

\bibitem{GL04a}
K.~Gr{\"o}chenig and M.~Leinert.
\newblock Symmetry and inverse-closedness of matrix algebras and functional
  calculus for infinite matrices.
\newblock {\em Trans. Amer. Math. Soc.}, 358(6):2695--2711 (electronic), 2006.

\bibitem{GRS10}
K.~Gr{\"o}chenig, Z.~Rzeszotnik, and T.~Strohmer.
\newblock Convergence analysis of the finite section method and {B}anach
  algebras of matrices.
\newblock {\em Integral Equations Operator Theory}, 67(2):183--202, 2010.

\bibitem{heil06}
C.~Heil.
\newblock Linear independence of finite {G}abor systems.
\newblock In {\em Harmonic analysis and applications}, Appl. Numer. Harmon.
  Anal., pages 171--206. Birkh\"auser Boston, Boston, MA, 2006.

\bibitem{HRT96}
C.~Heil, J.~Ramanathan, and P.~Topiwala.
\newblock Linear independence of time-frequency translates.
\newblock {\em Proc. Amer. Math. Soc.}, 124(9):2787--2795, 1996.

\bibitem{jaffard90}
S.~Jaffard.
\newblock Propri\'et\'es des matrices ``bien localis\'ees'' pr\`es de leur
  diagonale et quelques applications.
\newblock {\em Ann. Inst. H. Poincar\'e Anal. Non Lin\'eaire}, 7(5):461--476,
  1990.

\bibitem{Kur90}
V.~G. Kurbatov.
\newblock Algebras of difference and integral operators.
\newblock {\em Funktsional. Anal. i Prilozhen.}, 24(2):87--88, 1990.

\bibitem{Lin99}
P.~A. Linnell.
\newblock von {N}eumann algebras and linear independence of translates.
\newblock {\em Proc. Amer. Math. Soc.}, 127(11):3269--3277, 1999.

\bibitem{Sjo95}
J.~Sj{\"o}strand.
\newblock Wiener type algebras of pseudodifferential operators.
\newblock In {\em S\'eminaire sur les \'Equations aux D\'eriv\'ees Partielles,
  1994--1995}, pages Exp.\ No.\ IV, 21. \'Ecole Polytech., Palaiseau, 1995.

\end{thebibliography}

\end{document}